\RequirePackage{fix-cm}
\documentclass[smallextended]{svjour3}       
\smartqed  
\usepackage{graphicx}
 \usepackage{mathptmx}      
\usepackage{amssymb,amsmath}
\usepackage{amsmath,amscd}
\usepackage{amsthm}
\usepackage{overpic}
\usepackage{color}
\usepackage{standalone}
\usepackage{stmaryrd}

\usepackage{mathtools}
\usepackage{amsfonts}
\usepackage{mathrsfs}
\usepackage{tikz}
\usepackage[utf8]{inputenc}
\usepackage{graphicx}
\usepackage{comment}

\DeclareMathOperator{\skel}{\text{skel}_1}

\DeclarePairedDelimiter\floor{\lfloor}{\rfloor}

\newcommand{\N}{\mathbb{N}}
\newcommand{\Z}{\mathbb{Z}}

\newcommand{\R}{\mathbb{R}}

\newcommand{\al}{\alpha}

\newcommand{\G}{\Gamma}
\newcommand{\g}{\gamma}

\begin{document}

\title{Embedding Grid Graphs on Surfaces}


\author{Christian Millichap       \and
        Fabian Salinas
}


\institute{Christian Millichap \at
              Department of Mathematics, Furman University, Greenville, SC 29613 \\
              \email{Christian.Millichap@furman.edu}           
           \and
           Fabian Salinas \at
          Spartanburg, SC 29372 \\ 
	  \email{fabiansalinas125@gmail.com}
}

\date{Received: date / Accepted: date}

\maketitle

\begin{abstract}
In this paper, we analyze embeddings of grid graphs on orientable surfaces. We determine the genus of two infinite classes of $3$-dimensional grid graphs that do not admit quadrilateral embeddings and effective upper bounds for the genus of any $3$-dimensional grid graph, both in terms of a grid graph's combinatorics. As an application, we provide a complete classification of planar and toroidal grid graphs. Our work requires a variety of combinatorial and graph theoretic  arguments to determine effective lower bounds on the genus of a grid graph, along with explicitly constructing embeddings of grid graphs on surfaces to determine effective upper bounds on their genera. 
\begin{flushleft}
\textbf{Acknowledgment:} The Version of Record of this article is published in Graphs and Combinatorics, and is available online at https://link.springer.com/article/10.1007/s00373-022-02488-w
\end{flushleft}
\end{abstract}

\section{Introduction}
\label{sec:intro}

Define a \textbf{path of length $\alpha$} $\in 
\mathbb{N}$  as the graph $P_\al$ that has  vertex set $V= \{v\in \mathbb{N}: 0\leq v\leq \alpha\}$ and where two vertices determine an edge if and only if $|v_i - v_j| =1$ for $v_i, v_j \in V$.  A \textbf{$k$-dimensional grid graph} $G(\al_1,\ldots,\al_k)$ for $k\in \N$, is the graph Cartesian product 
$$G(\al_1,\ldots,\al_k) = P_{\al_1} \boxempty \cdots \boxempty P_{\al_k},$$
where $\al_i \in \N$ for index $1\leq i \leq k$.  We refer to the set $\{ \al_i \}_{i=1}^{k}$ as the \textbf{grid parameters} of $G(\al_1,\ldots,\al_k)$.  We  use $G(\ast)$ to denote a grid graph without specific grid parameters.  A graph $G$ \textbf{embeds} on a closed connected orientable surface of genus $g$, denoted $S_{g}$,  if there exists a continuous and one-to-one map $\phi: G \rightarrow S_{g}$. When it is obvious from context, we will let $G$ also stand for its image $\phi(G) \subset S_{g}$.    The \textbf{genus} of a graph $G$, denoted $\gamma(G)$, is the smallest $g$ such that $G$ embeds on $S_{g}$. Here, we are interested in studying embeddings of grid graphs and the genera of grid graphs.  

The genus of a grid graph $G(\ast)$ can be immediately determined when $G(\ast)$ is known to have a \textbf{quadrilateral embedding}, that is, an embedding where every face of $G(\ast)$ is bound by a $4$-cycle. Graphs admitting quadrilateral embeddings have been well studied; see \cite{AbPa1983}, \cite{An1981}, \cite{Pi1980}, \cite{Pi1982}, \cite{Pi1989}. In particular, White in \cite[Theorem 4]{Wh1970} proved that any grid graph $G(\ast)$ with at least three odd grid parameters admits a quadrilateral embedding, and so, determines the genus of $G(\ast)$. However, White's work does not address the infinite class of grid graphs that admit at most two odd grid parameters. This raises a number of questions. Do such grid graphs admit quadrilateral embeddings? If not, then can we determine their genera? In this paper, we focus on addressing these questions for $3$-dimensional grid graphs.  

We now provide an outline of our major results. In Proposition \ref{prop:3dcase3}, we detail an explicit construction used to find an upper bound on the genus of any $3$-dimensional grid graph with some even grid parameters, i.e., not covered by White's work. We then determine the genera for two infinite classes of $3$-dimensional grid graphs where the bound from Proposition \ref{prop:3dcase3} is sharp; see Theorem \ref{thm:genus} and Theorem \ref{thm:Da22}.  In addition, we show that any grid graph (with some even grid parameters) covered by these theorems, does not admit a quadrilateral embedding. Lower bounds for the the genera of the grid graphs discussed in Theorem \ref{thm:genus} and Theorem \ref{thm:Da22} are determined using block decompositions and finding complete bipartite graph minors in Section \ref{subsec:lowerbounds}. In Section \ref{sec:GenusRange}, we make a quick observation about the maximum genus of a grid graph, which is defined in that section. As a corollary of our work, we classify exactly which grid graphs are \textbf{planar}, i.e., embed on the plane and \textbf{toroidal}, i.e., embed on the torus, in Section \ref{sec:planarandtoroidal}.


\section{Background}
\label{sec:BlockandGenus}

To start, we review some  essential properties of grid graphs and graph embeddings. The following lemma is a basic application of well known vertex and edge counting formulas for Cartesian products of graphs. See \cite{ImKlRa2008} for more background on Cartesian products of graphs. We use $|V(G)|$ to denote the number of vertices in a graph $G$ and $|E(G)|$ to denote the number of edges in a graph $G$. 

\begin{lemma}
\label{prop:verticesforgrids}
Let $G(\al_1,\ldots,\al_k)$ be a $k$-dimensional grid graph. Then we have:
\begin{enumerate}
\item $|V(G(\al_1,\ldots,\al_k))| = \displaystyle\prod_{i=1}^{k} (\alpha_{i} +1).$ \vspace{0.1in}
\item $|E(G(\al_1,\ldots,\al_k))| =  \left(\displaystyle\prod_{i=1}^{k} (\alpha_{i} +1)\right) \left(\displaystyle\sum_{i=1}^k \frac{\al_i}{\al_i+1} \right).$
\end{enumerate}
\end{lemma}

 Lemma \ref{prop:verticesforgrids} provides a helpful dictionary which will be used in stating a number of our results on embeddings of grid graphs. An embedding of a graph $G$ on a surface $S$ is a \textbf{$2$-cell embedding} if each component of $S \setminus G$ is homeomorphic to an open disk, in which case, we say $G$ \textbf{2-cell embeds} on $S$. Note that, any $2$-cell embedding is also an embedding, though the converse isn't necessarily true.  We say that an embedding ($2$-cell embedding) of $G$ on $S_{g}$ is \textbf{minimal} if $\gamma(G) = g$.  We do not distinguish  between minimal embeddings and  minimal $2$-cell embeddings since the work of Youngs \cite{Yo1963} shows that each minimal embedding is a $2$-cell embedding. If $G$ $2$-cell embeds on a surface $S_{g}$, then it determines a $2$-cell decomposition of $S_{g}$ into vertices ($V$), edges ($E$), and faces ($F$), where $V$ and $E$ correspond with the vertices and edges of $G$, respectively, and each face is a component of $S_g \setminus G$. Given such a $2$-cell embedding,  the Euler characteristic of $S_g$ is exactly $\chi(S_g) = |V| - |E| + |F|$. 
 
 The following proposition and corollary motivate why we first consider the genera of quadrilateral grid graphs. Both results are certainly well known but we include them here for completeness.  Recall that the \textbf{girth} of a graph $G$ (that is not a tree) is the length of its shortest cycle.

\begin{proposition}
\label{prop:Eulergirth}
Suppose $G$ is a  connected graph whose girth is at least $4$. Then $$\gamma(G) \geq 1 + \frac{|E|}{4} - \frac{|V|}{2}.$$ 
\end{proposition}

\begin{proof}
Suppose $G$ $2$-cell embeds on a surface $S_g$, for some $g \in \mathbb{N}_0$. The boundary of a face of this $2$-cell embedding is a closed walk $W$ in $G$. We claim that $H$, the subgraph of $G$ induced by $W$, contains a cycle. To see this, note that if $H$ did not contain a cycle, then $H$ would be a tree. In this case, $G=H$ since our $2$-cell embedding could only have one face if $H$ is a tree. However, this violates the assumptions that $G$ has a cycle since it has nontrivial girth. Thus, the boundary of each face in this $2$-cell embedding contains a subgraph that is a cycle, which must contain at least $4$ edges based on the girth of $G$.  At the same time, each edge of  this $2$-cell embedding appears on the boundary of at most two faces. Thus, $|E| \geq \frac{4|F|}{2} = 2|F|$, and so, $|F| \leq \frac{|E|}{2}$. Then
\begin{equation*}
2-2g = \chi(S_g)  = |V| - |E| + |F| \leq |V| - |E| + \frac{|E|}{2} = |V| - \frac{|E|}{2},
\end{equation*}   
 This implies that $g \geq 1 + \frac{|E|}{4} - \frac{|V|}{2}$ and since this is true for any $2$-cell embedding of $G$ on a closed connected orientable surface, we have that $\gamma(G) \geq 1 + \frac{|E|}{4} - \frac{|V|}{2}.$
\end{proof}

\begin{corollary}
\label{cor:Eulergirth}
Let $G(\al_1, \ldots, \al_k)$ be a grid graph with $k>1$. Then  $$\gamma(G(\al_1, \ldots, \al_k)) \geq  1 + \frac{|E|}{4} - \frac{|V|}{2} =    1 + \frac{1}{2}\prod_{i=1}^{k} (\alpha_{i} +1)\Big[ \frac{1}{2}\sum_{i=1}^{k} \frac{\al_i}{\al_i +1} -1 \Big]$$ 
We have equality in the above formula if and only if $G(\al_1, \ldots, \al_k)$ admits a quadrilateral embedding.
\end{corollary}

\begin{proof}
First, we show that the girth of a $k$-dimensional grid  graph is exactly 4 whenever $k>1$. It is easy to see that any such grid graph $G(\ast) = G(\al_1, \ldots, \al_k)$ with $k >1$ contains a subgraph isomorphic to $G(1,1)$, and so, $G(\ast)$ will always contain a $4$-cycle. In addition, any grid graph is a simple graph (contains no loops or multi-edges), so we only need to show that $G(\ast)$ has no $3$-cycles. Suppose $G(\ast)$ has a $3$-cycle determined by the vertices $v_i \in V(G(\ast))$ for $i=1,2,3$ with edges connecting each pair of vertices.  Note that, there exists an edge between two vertices in a grid graph if and only if the corresponding $k$-tuples differ by one in one component and are equal in the rest. Thus,  the $k$-tuples representing $v_1$ and $v_2$ differ by one in exactly one grid parameter and are equal in the rest of their respective grid parameters. Similarly for the pair $v_1$ and $v_3$ and the pair $v_2$ and $v_3$. However, this is not possible if $v_1$, $v_2$, and $v_3$ are all distinct vertices. Thus, the girth of $G(\ast)$ is at least 4. 

Now, the desired inequality in the first statement immediately follows  from combining Proposition \ref{prop:Eulergirth} with the vertex and edge counting formulas from Lemma \ref{prop:verticesforgrids}. 

The second statement essentially follows from the proof of Proposition \ref{prop:Eulergirth}. First, note that every edge in a grid graph occurs on a cycle. Then the inequality $|F| \leq \frac{|E|}{2}$ in the proof of Proposition \ref{prop:Eulergirth} becomes an equality if and only if every face in a $2$-cell decomposition is bound by a $4$-cycle. 
\end{proof}

 Corollary \ref{cor:Eulergirth} shows that in some sense, quadrilateral grid graphs realize the smallest possible genus for any such embedding of a grid graph relative to the number of vertices and edges of that grid graph.  The following theorem shows that many grid graphs admit quadrilateral embeddings, and so, we can determine their respective genera. This result was first proved by White  in \cite[Theorem 4]{Wh1970}. Here, we have rephrased his work  in terms of our notation.

\begin{theorem}\cite[Theorem 4]{Wh1970}
\label{thm:kdimgenus1}
Let $G(\al_1,  \ldots, \al_k)$ be a $k$-dimensional grid graph where $k \geq 3$ and at least three $\al_i$ are odd, for $i =1, \ldots, k$. Then $G(\al_1,  \ldots, \al_k)$ admits a quadrilateral embedding, i.e., 
$$\gamma(G(\al_1,  \ldots, \al_k))  = 1 + \frac{|E|}{4} -\frac{|V|}{2}   = 1 + \frac{1}{2}\prod_{i=1}^{k} (\alpha_{i} +1)\Big[ \frac{1}{2}\sum_{i=1}^{k} \frac{\al_i}{\al_i +1} -1 \Big].$$
\end{theorem}


\section{Upper bounds on the genus of a grid graph}
\label{subsec:3dgenus}

In this section, we  construct a $2$-cell embedding  for each $3$-dimensional grid graph $G(\ast)$, providing an upper bound on the genus of $G(\ast)$. In Subsection \ref{subsec:allodd}, we discuss the case where all three grid parameters are odd. By White's work \cite[Theorem 4]{Wh1970}, such grid graphs are known to be quadrilateral and his proof provides an explicit construction for such an embedding. In Subsection \ref{subsec:allodd}, we provide our own method for building a quadrilateral embedding for such grid graphs. Building off of this construction, we then consider the case where some grid parameters are even in Subsection \ref{subsec:upperbounds}.

\subsection{$G(\al_1, \al_2, \al_3)$ where all grid parameters are odd}
\label{subsec:allodd}

We first describe how to explicitly construct surfaces with a specified $2$-cell structure. We will see that certain $3$-dimensional grid graphs embed on these surfaces as the $1$-skeleton of their $2$-cell structure. Let $U=\{(x_1,x_2,x_3)\in \R^3: 0 \leq x_i \leq 1 \text{ for } i=1,2,3\}$ be the unit cube with it's natural cell decomposition into eight $0$-cells, twelve $1$-cells, six $2$-cells, and one $3$-cell. Let $T(n_1,n_2,n_3)$ be the unit cube translated by the vector $(n_1,n_2,n_3)\in \N_0^3$ and assume $T(n_1, n_2, n_3)$ inherits this $3$-cell structure from $U$. 
Define $C=C(\al_1,\al_2,\al_3)$  for $(\al_1,\al_2,\al_3)\in \N^3$ as 
$$\begin{aligned} C = \{ & (n_1,n_2,n_3)\in \N_0^3: \text{ At least 2 components in $(n_1,n_2,n_3)$ are even,} \\ & \text{ and } n_i < \al_i \text{, for } i = 1,2,3\}. \end{aligned}$$ 

Define  $W(\al_1,\al_2,\al_3)$ as:
$$ W(\al_1,\al_2,\al_3) = \bigcup_{(n_1,n_2,n_3)\in C} T(n_1,n_2,n_3).$$
See Figure \ref{fig:3d-grid-surface} for a visualization of $W(\al_1, \al_2, \al_3)$ where each $\al_i$ is odd, for $i=1, 2, 3$. Intuitively, $W(\al_1, \al_2, \al_3)$ is constructed by gluing together copies of unit cubes along some of their $2$-cells in pairs, where $C$ provides instructions for where to place these cubes in $\mathbb{R}^{3}$ and which $2$-cells to glue together. Based on this construction, it's not difficult to see that $W(\al_1, \al_2, \al_3)$ is a $3$-manifold with boundary. Furthermore, $W(\al_1, \al_2, \al_3)$ inherits a $3$-cell structure, $\mathcal{T}_W$, from the collection of translated cubes,  $\{ T(n_1, n_2, n_3)\}$, used in its construction.

Now  let $S(\al_1, \al_2, \al_3) = \partial W(\al_1, \al_2, \al_3)$, which must be a closed surface since it is the boundary of a $3$-manifold.  $S(\al_1, \al_2, \al_3)$ is  connected since one can easily construct a continuous path from any point in $S(\al_1, \al_2, \al_3)$ to $(0, 0, 0) \in S(\al_1, \al_2, \al_3)$; we leave the details to the reader.   In addition, one can determine that $S(\al_1, \al_2, \al_3)$ is orientable by subdividing each  $2$-cell (all of which are squares) into a pair of triangles and providing consistent orientations across this triangulation. Thus, $S(\al_1, \al_2, \al_3)$ is homeomorphic to $S_g$ for some $g \in \mathbb{N} \cup \{0 \}$. The $3$-cell structure on $W(\al_1, \al_2, \al_3)$  induces a $2$-cell structure on $S(\al_1, \al_2, \al_3)$, which we denote by $\mathcal{T}_{S}$. 
By construction, every face ($2$-cell) of $\mathcal{T}_S$ is bounded by a $4$-cycle.

We define the 1-skeleton of a cell complex, $\Sigma$,  as the union of its $0$-cells and $1$-cells, and refer to this set as $\skel(\Sigma)$. We claim that $$\skel(\mathcal{T}_{S}) = \bigcup_{(n_1,n_2,n_3)\in C} \skel(T(n_1,n_2,n_3)).$$  Suppose there existed a point $x\in\skel(T(n_1,n_2,n_3))$ where $(n_1,n_2,n_3)\in C$ such that $x$ belonged to the interior of $W(\al_1,\al_2,\al_3)$. It must follow then that $x$ is part of the boundary of at least 4 translated unit cubes $T_1,T_2,T_3,T_4\subseteq W(\al_1,\al_2,\al_3)$. However, the vectors that determine these unit cubes cannot all be in $C(\al_1,\al_2,\al_3)$, which implies one of these unit cubes is not in $W(\al_1,\al_2,\al_3)$. Thus, it must be the case that every edge in $\skel(T(n_1,n_2,n_3))$ must lie on the boundary of $W(\al_1,\al_2,\al_3)$. By definition of $\mathcal{T}_{S}$, we have that 
$$\bigcup_{(n_1,n_2,n_3)\in C} \skel(T(n_1,n_2,n_3))\subseteq \skel(\mathcal{T}_{S}).$$ 
Inclusion in the other direction follows trivially from the definitions. In the following proposition, we will show that if $\al_1$, $\al_2$, and $\al_3$ are all odd, then $G(\al_1, \al_2, \al_3)$ embeds on $S(\al_1, \al_2, \al_3)$ as $\skel(\mathcal{T}_{S})$.

\begin{figure}[ht]
\centering
\includegraphics[width=\textwidth]{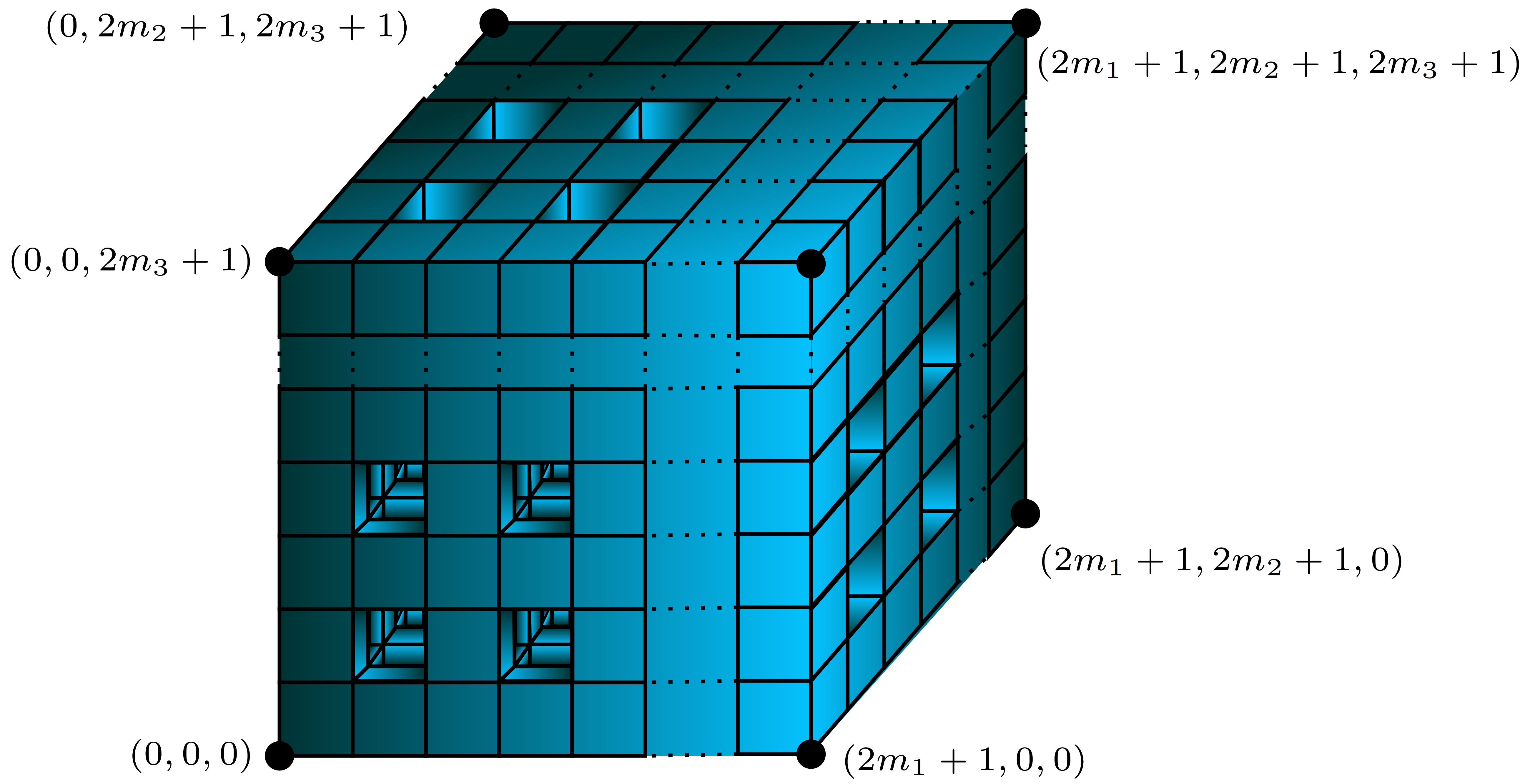}
\caption{A representation of $W(\al_1,\al_2,\al_3)$ when $\al_1,\al_2,\al_3$ are all odd.}
\label{fig:3d-grid-surface}
\end{figure}

\begin{proposition}
\label{prop:3dalloddgenus}
Let $G(\al_1,\al_2,\al_3)$ be a $3$-dimensional grid graph where all three grid parameters are odd. Then $G(\al_1,\al_2,\al_3)$ admits a quadrilateral embedding, i.e., 
$$\begin{aligned} \gamma(G(\al_1,\al_2,\al_3)) & = 1 + \frac{|E(G(\al_1,\al_2,\al_3))|}{4} -\frac{|V(G(\al_1,\al_2,\al_3))|}{2}. \\ \end{aligned}$$
\end{proposition}

\begin{proof}
Let $G(\ast) = G(\al_1, \al_2, \al_3)$ be a $3$-dimensional grid graph with all odd grid parameters. Corollary \ref{cor:Eulergirth} tells us that $\gamma(G(\ast)) = 1 + \frac{|E(G(\ast))|}{4} - \frac{|V(G(\ast))|}{2}$ if and only if $G$ admits a $2$-cell embedding on a surface where every face is bound by a $4$-cycle in $G(\ast)$. We now describe how to construct a $2$-cell embedding of $G(\ast)$ on a surface $S$ that meets these qualifications. Let $S= S(\al_1,\al_2,\al_3)$ be as defined above with corresponding $2$-cell structure $\mathcal{T}_{S}$.

Consider the following subset of $\R^3$:
$$ \begin{aligned} Y = & \{(x_1,x_2,x_3)\in \R^3: \text{ At least 2 components in $(x_1,x_2,x_3)$ are integers,} \\ & \text{ and } 0\leq x_i \leq \al_i \text{, for } i = 1,2,3\}. \end{aligned}$$

It isn't too difficult to see that $Y$ is an embedding of $G(\ast)$ in $\R^3$ where vertices of $G(\ast)$ are mapped to their respective coordinates in $\mathbb{R}^{3}$ and each edge of $G(\ast)$ is mapped to a straight line segment.

We first show $\skel(\mathcal{T}_{S})\subseteq Y$. Recall that $\skel(\mathcal{T}_{S})$ is constructed by taking the union of 1-skeletons of the translated unit cubes used in the construction of $W(\al_1,\al_2,\al_3)$. Given how we defined the unit cube $U$, it follows that any $(u_1,u_2,u_3)\in \skel(U)$ must have that at least 2 components in $(u_1,u_2,u_3)$ that are integers and $0\leq u_i \leq 1\leq \al_i \text{ for } i = 1,2,3$. By the definition of $Y$, we can deduce that $\skel(U)\subseteq Y$. 

Now consider any $T(n_1,n_2,n_3)$ with $(n_1,n_2,n_3)\in C = C(\al_1,\al_2,\al_3)$. Fix $(x_1,x_2,x_3)\in \skel(T(n_1,n_2,n_3))$, and note that $(x_1,x_2,x_3) = (u_1+n_1,u_2+n_2,u_3+n_3)$ for some $ (u_1,u_2,u_3)\in \skel(U)$. By definition of $C$, we have that $0\leq n_i < \al_i$ and $n_i\in \N$ for $i = 1,2,3$. Since $0\leq u_i \leq 1 \text{ for } i = 1,2,3$, we have that
$$0 \leq u_i+n_i \leq \al_i \implies 0 \leq x_i \leq \al_i, \text{ for } i = 1,2,3.$$

Also, since at least two components in $(u_1,u_2,u_3)$ must be integers (otherwise $(u_1,u_2,u_3)\notin\skel(U)$) and $n_1, n_2, n_3 \in \mathbb{Z}$, we can deduce that at least two components in $(u_1+n_1,u_2+n_2,u_3+n_3) = (x_1,x_2,x_3)$ are also integers. So then by definition of $Y$, $(x_1,x_2,x_3)\in Y$. Since the argument holds for any $(x_1,x_2,x_3)\in \skel(T(n_1,n_2,n_3))$, we have $\skel(T(n_1,n_2,n_3))\subseteq Y$. Since the argument holds for any $(n_1,n_2,n_3)\in C$ and $\skel(U)\subseteq Y$, we have that
$$\skel(\mathcal{T}_{S}) = \bigcup_{(n_1,n_2,n_3)\in C} \skel(T(n_1,n_2,n_3)) \subseteq Y,$$

as desired. 
We now must show that $Y\subseteq \skel(\mathcal{T}_{S})$. Let $(x_1,x_2,x_3)\in Y$. Without loss of generality, suppose $x_1 = n_1$ and $x_2 = n_2$ where $n_1,n_2\in \Z$. Also, let $n_3\in \N_0$ be the largest such that $n_3\leq x_3$ and $n_3<\al_3$. By definition of $Y$, we can deduce that $0 \leq n_i  \leq \al_i$ for $i = 1,2,3$. Now consider $\left(2\left\lfloor\frac{n_1}{2}\right\rfloor ,2\left\lfloor \frac{n_2}{2}\right\rfloor, n_3 \right) \in \N_0^3$. We aim to show that $\left(2\left\lfloor\frac{n_1}{2}\right\rfloor ,2\left\lfloor \frac{n_2}{2}\right\rfloor, n_3 \right)\in C$, and 
$$(x_1,x_2,x_3)\in \skel(T\left(2\left\lfloor\frac{n_1}{2}\right\rfloor ,2\left\lfloor \frac{n_2}{2}\right\rfloor, n_3 \right)) \subseteq \skel(\mathcal{T}_{S}).$$

Since by assumption $\al_1$ and $\al_2$ are both odd, we have
$$0 \leq n_i  \leq \al_i \implies 0 \leq 2\left\lfloor \frac{n_i}{2}\right\rfloor  < \al_i, \text{ for } i =1,2.$$

Clearly, $\left(2\left\lfloor\frac{n_1}{2}\right\rfloor ,2\left\lfloor \frac{n_2}{2}\right\rfloor, n_3 \right)$ contains 2 even components and since  we assumed $n_3< \al_3$, we have by the definition of $C$ that 
$$\left(2\left\lfloor\frac{n_1}{2}\right\rfloor ,2\left\lfloor \frac{n_2}{2}\right\rfloor, n_3 \right)\in C.$$

Now depending on the parity of $n_1$ and $n_2$, we have that $(x_1,x_2,x_3) = (n_1,n_2,x_3)$ lies on exactly one of four 1-simplices in $T\left(2\left\lfloor\frac{n_1}{2}\right\rfloor ,2\left\lfloor \frac{n_2}{2}\right\rfloor, n_3 \right))$. Figure \ref{Translated-cube-figure} illustrates this.

\begin{figure}[ht]
\centering
\includegraphics[width=\textwidth]{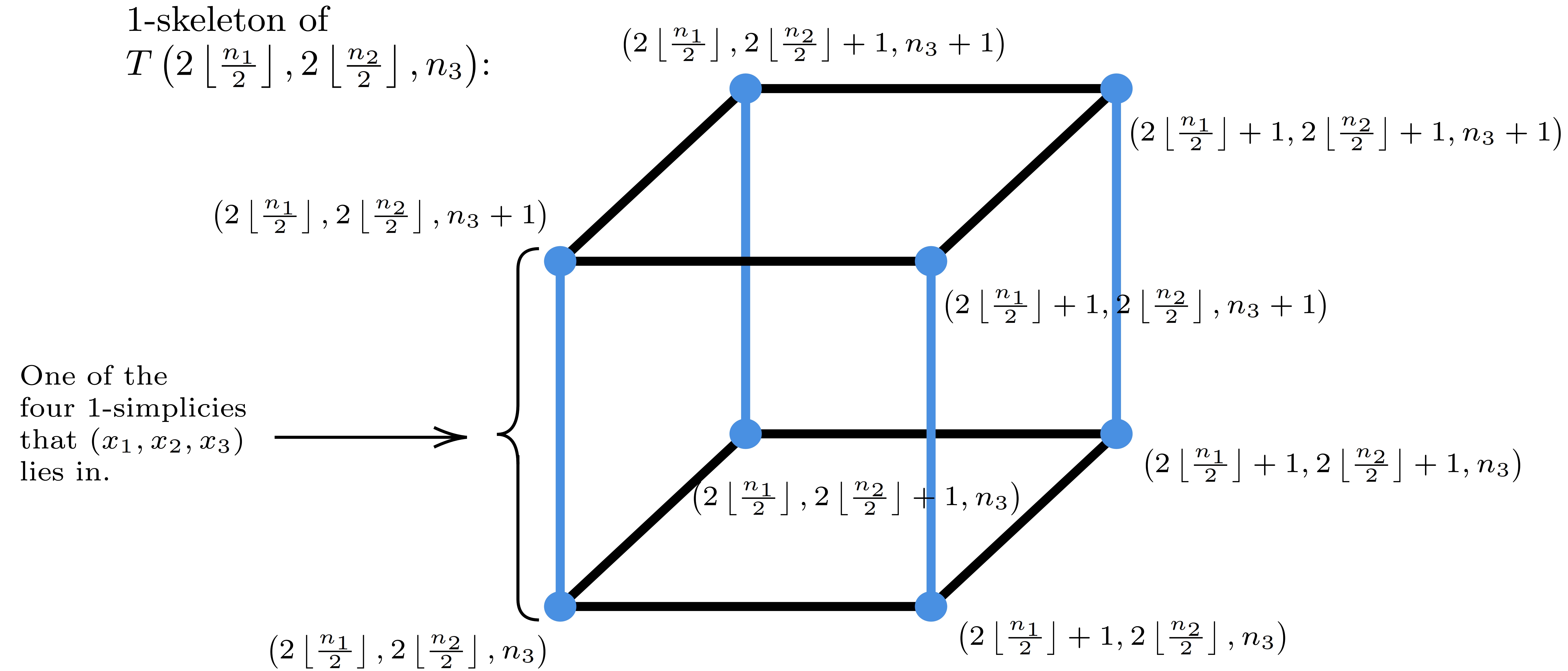}
\caption{The 1-skeleton of the translated unit cube $T\left(2\left\lfloor\frac{n_1}{2}\right\rfloor ,2\left\lfloor \frac{n_2}{2}\right\rfloor, n_3 \right))$. The parity of $n_1$ and $n_2$ determine which 1-simplex $(x_1,x_2,x_3)$ lies in and how $n_3$ was defined implies $n_3\leq x_3\leq n_3+1$.}
\label{Translated-cube-figure}
\end{figure}

Since the argument holds for any $(x_1,x_2,x_3)\in Y$, we have that $$Y\subseteq \skel(\mathcal{T}_{S}).$$

In summary, we have $Y$ is an embedding of $G(\ast)$ in $\mathbb{R}^{3}$ with $Y = \skel(\mathcal{T}_{S})$. By construction, $\skel(\mathcal{T}_{S})$ is the one-skeleton of the $2$-cell structure $\mathcal{T}_{S}$ for the surface $S$, where each face bounds a $4$-cycle, completing the proof. \end{proof}

The construction used in the proof of Proposition \ref{prop:3dalloddgenus} cannot  immediately be extended to $3$-dimensional grid graphs where some grid parameters are even. It is always the case that the $1$-skeleton of $\mathcal{T}_{S(\al_1,\al_2,\al_3)}$ will provide a $2$-cell embedding of a graph $G$ on the surface $S(\al_1, \al_2, \al_3)$ where every face bounds a $4$-cycle. However, if at least one of the grid parameters is even, then the graph $G$ will not  be isomorphic to the desired grid graph. For example, one could construct the surface $S(2,2,2)$ with its corresponding $2$-cell complex $\mathcal{T}_{S(2,2,2)}$. Following the proof of Proposition \ref{prop:3dalloddgenus}, one can still show that $\skel(\mathcal{T}_{S(2,2,2)}) \subseteq Y$ where $Y$ is the embedding of $G(2,2,2)$ in $\R^3$ outlined in the proof. 

Trouble arises, however, in showing that $Y \subseteq \skel(\mathcal{T}_{S(2,2,2)})$, since $0 \leq 2\left\lfloor \frac{n_i}{2}\right\rfloor  < \al_i$ does not hold for $n_i = \al_i = 2$. Effectively, this implies that $\skel(\mathcal{T}_{S(2,2,2)})$ is a proper subgraph of $G(2,2,2)$; similar arguments hold for any grid graph with even components.

\subsection{$G(\al_1, \al_2, \al_3)$ where some grid parameters are even}
\label{subsec:upperbounds}

Here, we establish effective upper bounds on the genus of any $3$-dimensional grid graph where some grid parameter is even. Afterwards, we briefly describe how to construct a general (ineffective) upper bound on the genus of any $k$-dimensional grid graph.

Let $G(\ast)$ be a grid graph and let $I \subset V(G(\ast))$. We say that $G_I$ is the \textbf{subgraph of $G(\ast)$ induced by $I$} where $V(G_I) = I$ and $E(G_I)$ is the set of all edges in $G(\ast)$ that connect two vertices in the set $I$.

\begin{proposition}
\label{prop:3dcase3}
If $G(\al_1, \al_2, \al_3)$ be a $3$-dimensional grid graph with at least one even grid parameter. \begin{enumerate}
\item If exactly one grid parameter is even (the third parameter), then $$\gamma(G(\al_1, \al_2, \al_3)) \leq \gamma(G(\al_1, \al_2, \al_3 -1)) + \frac{(\al_1 +1)(\al_2+1)}{4} -1.$$
\item If exactly two grid parameters are even (the second $\&$ third parameters), then  $$ \gamma(G(\al_1, \al_2, \al_3)) \leq \gamma(G(\al_1, \al_2, -1, \al_3 -1)) + \frac{(\al_1 +1)(\al_2 + \al_3)}{4}   - 1.$$
\item If all three grid parameters are even, then $$\gamma(G(\al_1, \al_2, \al_3)) \leq  \gamma(G(\al_1 -1, \al_2 -1, \al_3 -1)) + \frac{\al_1\al_2 + \al_1\al_3 + \al_2\al_3}{4}-1.$$
\end{enumerate}
\end{proposition}

\begin{proof}
Here, we provide full details for proving the third statement. The first two statements follow from similar arguments and brief details for these are given at the end of the proof.

Suppose $G(\ast) = G(\al_1, \al_2, \al_3)$ is a $3$-dimensional grid graph with all three grid parameters even. Let 
$$I = \{ (\beta_1, \beta_2, \beta_3)  :  0 \leq \beta_i \leq \al_i -1 , \hspace{0.05in} \text{for all} \hspace{0.05in} i=1,2,3  \} \subset V(G(\ast))$$ and let 
$$J = \{ (\beta_1, \beta_2, \beta_3)   :  \beta_i = \al_i \hspace{0.05in} \text{for some} \hspace{0.05in} i=1,2,3 \} \subset V(G(\ast)).$$  Consider the induced subgraphs $G_I$ and $G_J$ of $G(\ast)$. It is easy to see that $G_I \cap G_J = \emptyset$,  $V(G(\ast)) = V(G_I) \cup V(G_J)$, and  $G_I \cong G(\al_1 -1, \al_2 -1, \al_3 -1)$. Let $E_B = E(G(\ast)) \setminus \Big(E(G_I) \cup E(G_J)\Big)$. By construction, we have that $G(\ast) = G_I \cup G_J \cup E_B$.  

Since $G_I$ is a $3$-dimensional grid graph with all odd parameters, we can $2$-cell embed $G_I$ on a surface $S_I$ whose genus is $\gamma(G_I)$ where every face bounds a $4$-cycle, as done in the proof of Proposition \ref{prop:3dalloddgenus}. In particular, the embedding of $(G_I, S_I)$ can be realized in $\mathbb{R}^{3}$ where vertices of $G_I \cong G(\al_1 -1, \al_2 -1, \al_3 -1)$ are placed at their corresponding coordinates in $\mathbb{R}^{3}$ and where edges are embedded as straight line segments between these vertices.

We now consider a particular decomposition and embedding of $G_J$. First off, it isn't difficult to see that $G_J$ $2$-cell embeds on a $2$-sphere $S_J \subset \mathbb{R}^{3}$ that is disjoint from $(G_I, S_I) \subset \mathbb{R}^{3}$. Consider $G_J$ embedded in $\mathbb{R}^{3}$ with each vertex of $G_J$ placed at its respective coordinate and where edges  of $G_J$ are embedded as straight line segments between vertices.  Insert a $2$-cell bounding each $4$-cycle in $G_J$ to create a $2$-cell complex that will be homeomorphic to a $2$-disk, $D_1$. Attach another (topological) $2$-disk, $D_2$, to $D_1$ glued along their respective boundaries, where $D_2$ is embedded in $\mathbb{R}^{3}$ and is disjoint from the surface $S_{I}$ (many such $2$-disks exist) to create the $2$-sphere $S_J$ with the necessary conditions.

Let $(G_J, S_J)$ denote this particular embedding of $G_J$ and let 
$$J_1 = \{ (\al_1, \beta_2, \beta_3)  :  0 \leq \beta_i \leq \al_{i}-1 \hspace{0.05in} \text{for i} \hspace{0.05in} = 2, 3 \} \subset V(G_J) \subset V(G),$$
$$J_2 = \{ (\beta_1, \al_2, \beta_3)  :  0 \leq \beta_i \leq \al_{i}-1 \hspace{0.05in} \text{for i} \hspace{0.05in} = 1, 3 \}\subset V(G_J) \subset V(G),$$
$$J_3 = \{ (\beta_1, \beta_2, \al_3)  :  0 \leq \beta_i \leq \al_{i}-1 \hspace{0.05in} \text{for i} \hspace{0.05in} = 1, 2 \}\subset V(G_J) \subset V(G).$$
We consider the induced subgraphs for these respective sets of vertices and note that $G_{J_{1}} \cong G(\al_2 -1, \al_3 -1)$, $G_{J_{2}} \cong G(\al_1 -1, \al_3 -1)$, and $G_{J_{3}} \cong G(\al_1 -1, \al_2 -1)$.   

To construct an embedding of $G(\ast)$ on a surface $S$ with the required genus, we will take $(G_I, S_I)$ and $(G_J, S_J)$ and add in handles to connect these two surfaces while adding in the edges from $E_B$ along these handles to connect $G_I$ to $G_J$ in the required manner. By construction, $E_B$ must connect a vertex from $G_J$ to a vertex in $G_I$. By definition of grid graph, this implies that for any edge $[v_i, v_j] \in E_B$ with say $v_i \in G_I$ and $v_j \in G_J$, we must have that $v_i$ and $v_j$ have the same value for two of their grid parameters while the remaining grid parameter is $\al^* -1$ for $v_i$ and $\al^*$ for $v_j$, where $\al^*$ could be either $\al_1, \al_2,$ or $\al_3$. Thus, there is a bijective correspondence between $E_B$ and $J_1 \cup J_2 \cup J_3$. Furthermore, for each vertex  $v_i = (\al_1, \beta_2, \beta_3) \in J_1$, there is a unique edge in $E_B$ connecting $v_i$ to the vertex $v_j = (\al_1 -1, \beta_2, \beta_3) \in V(G_I)$, and similarly done for $J_2$ and $J_3$. To add in the edges of $E_B$ that connect  $G_{J_{1}} \subset G_J$ to $G_I$, remove the interiors of the faces of $(G_J, S_J)$ centered at $(\al_1, 1/2+2m, 1/2+2n)$ where $m=0,\ldots, \frac{\al_2}{2}-1$ and $n = 0, \ldots, \frac{\al_3}{2}-1$ Likewise, remove the interiors of the corresponding face on $(G_I, S_I)$, i.e. remove the interiors of the faces centered at $(\al_1 -1, 1/2+2m, 1/2+2n)$ where $m=0,\ldots, \frac{\al_2}{2}-1$ and $n = 0, \ldots, \frac{\al_3}{2}-1$; by construction, the $2$-cell complex for $S_I$ has faces centered at all of these coordinates.  Now attach a handle (a copy of $\mathbb{S}^{1} \times [0,1]$) between corresponding faces that have been removed, where $\mathbb{S}^{1} \times \{ 0 \}$ glues to the face on $G_{J_{1}}$ and $\mathbb{S}^{1} \times \{ 1 \}$ glues to the corresponding face on $G_I$.  At the same time, we can embed a set of four edges from $E_B$ along each such handle, where each  edge connects from a vertex $v_j \in J_1$ to the corresponding vertex $v_I \in I$ as described above.  This will result in adding $(\frac{\al_2}{2})(\frac{\al_3}{2})$ handles between $(G_I, S_I)$ and $(G_J, S_J)$. A similar procedure can be performed for $J_2$ and $J_3$, which results in  adding $(\frac{\al_1}{2})(\frac{\al_3}{2})$ and $(\frac{\al_1}{2})(\frac{\al_2}{2})$ handles, respectively, between $(G_I, S_I)$ and $(G_J, S_J)$. At the end of this process, we will have a $2$-cell embedding of $G(\ast)$ on a surface $S$ with genus

$$\begin{aligned} \gamma(S) &= \gamma(S_I) + \gamma(S_J) + \Big(\frac{\al_2}{2}\Big)\Big(\frac{\al_3}{2}\Big) +\Big(\frac{\al_1}{2}\Big)\Big(\frac{\al_3}{2}\Big) + \Big(\frac{\al_1}{2}\Big)\Big(\frac{\al_2}{2}\Big) -1 \\
&= \gamma(G(\al_1 -1, \al_2 -1, \al_3 -1)) + \frac{\al_1\al_2 + \al_1\al_2 + \al_2\al_3}{4}-1.
\end{aligned}$$

We subtract $1$ at the end of this formula since the first handle attached connecting $S_I$ to $S_J$ does not contribute to the genus of the resulting surface. 

For items (1) and (2), a similar decomposition of $G(\ast)$ can be used. For instance, to prove (1), let $I = \{ (\beta_1, \beta_2, \beta_3) \hspace{0.05in} : \hspace{0.05in} 0 \leq \beta_1 \leq \al_1, 0 \leq \beta_2, \leq \al_2, 0 \leq \beta_3 \leq \al_3 -1 \} \subset V(G)$ and let $J = \{ (\beta_1, \beta_2, \al_3) \hspace{0.05in} : \hspace{0.05in} 0 \leq \beta_1 \leq \al_1, 0 \leq \beta_2 \leq \al_2 \} \subset V(G).$ Consider the induced subgraphs $G_I$ and $G_J$ of $G(\ast)$, and define $E_B = E(G(\ast)) \setminus \Big(E(G_I) \cup E(G_J)\Big)$. By construction, we have that $G(\ast) = G_I \cup G_J \cup E_B$. Again, use the $2$-cell embedding of $G_I$ on $S_I$ coming from Proposition \ref{prop:3dalloddgenus} and use the $2$-cell  embedding of $G_J$ on a $2$-sphere $S_J$ in the same manner as the previous case.  Now, add handles between $(S_I, G_I)$ and $(S_J, G_J)$ that contain the edges of $E_B$ in order to construct a $2$-cell embedding of $G(\ast)$ on a surface with the required genus.  
\end{proof}

\begin{remark}
By Corollary \ref{cor:Eulergirth}, the upper bounds in Proposition \ref{prop:3dcase3} can be computed  in terms of the grid parameters of the given grid graph. 
\end{remark}

\begin{remark}
The upper bounds for genus given in Proposition \ref{prop:3dcase3} are sharp for infinite families of grid graphs; see the proofs of Theorem \ref{thm:genus} and  Theorem \ref{thm:Da22}. In fact, we conjecture that this upper bound is always sharp; see Conjecture \ref{conjecture1}.
\end{remark}

The following corollary gives an alternative method for calculating the genus of any surface constructed in the proof of  Proposition \ref{prop:3dcase3}.  In what follows, let $P = 0$ if $G(\ast)$ has no even grid parameters, $P = 2\al_1 + 2\al_2$ if $G(\ast)$ has one even grid parameter (the third parameter), and  $P = 2\al_1 + 2\al_2 + 2\al_3$ if $G(\ast)$ has two or three even grid parameters. Intuitively,  $P$ stands for the perimeter length of a particular face constructed in the $2$-cell embedding from Proposition \ref{prop:3dcase3}.

\begin{corollary}
Let $G(\al_1, \al_2, \al_3)$ be a $3$-dimensional grid graph. Then $G(\al_1, \al_2, \al_3)$ $2$-cell embeds on a surface $S$ with genus $$\gamma(S) = \frac{1}{2} + \frac{|E(G(\al_1, \al_2, \al_3))|}{4}  - \frac{|V(G(\al_1, \al_2, \al_3))|}{2}  + \frac{P}{8}.$$
\end{corollary}

\begin{proof}
First, if all three grid parameters are odd then $P=0$ and the result immediately follows from Proposition \ref{prop:3dalloddgenus}. From here, we follow the same procedure used in the proof of Corollary \ref{cor:Eulergirth} while taking into account the specific $2$-cell embeddings built in the proof of Proposition \ref{prop:3dcase3}. Let $(G(\ast),S_g)$ be a $2$-cell embedding of $G(\ast) = G(\al_1, \al_2, \al_3)$ on the surface $S_g$ from the proof of Proposition \ref{prop:3dcase3}. By construction, every face in this $2$-cell embedding bounds a 4-cycle, except the one ``outer'' face $F^*$ on the $2$-sphere $S_J$; this face is homeomorphic to the $2$-disk $D_1$ used to construct $S_J$ in the proof of Proposition \ref{prop:3dcase3}. The cycle length of the boundary of $F^*$ is exactly $P$. Since every edge bounds exactly $2$-faces, we have that $|E(G(\ast))| = \frac{4(|F|-1)}{2} + \frac{P}{2}$, which implies that $|F| = \frac{|E(G(\ast))|}{2} - \frac{P}{4} +1$. Thus,

$$\begin{aligned} 2 - 2g & = \chi(S_g) = |V(G(\ast))| - |E(G(\ast))| + \Big(\frac{|E(G(\ast))|}{2} - \frac{P}{4} +1\Big) \\ &= |V(G(\ast))| - \frac{|E(G(\ast))|}{2} - \frac{P}{4} +1. \end{aligned}$$

Solving in terms of $g$ gives the desired result. 
\end{proof}

Tubing and pasting arguments similar to the ones implemented in White's work \cite{Wh1970}, could possibly be used to obtain upper bounds on the genus of any $k$-dimensional grid graph. Here, we provide a quick (ineffective) upper bound that certainly could be improved upon.

\begin{lemma}
\label{lem:badgenusupperbound}
Let $G(\al_1, \ldots, \al_k)$ be any $k$-dimensional grid graph. Then
$$ \gamma(G(\al_1, \ldots, \al_k)) \leq \gamma(G(\al_1 -1, \al_2, \ldots, \al_k)) + \gamma(G(\al_2, \ldots, \al_k)) + \Big[\prod_{i=2}^{k} (\al_i +1) \Big] -1.$$
\end{lemma}

\begin{proof}
Given a $k$-dimensional grid graph $G(\ast)= G(\al_1, \ldots, \al_k)$, let $$I = \{ (\beta_1, \beta_2, \ldots, \beta_k)  :  0 \leq \beta_1 \leq \al_1 -1, 0 \leq \beta_i, \leq \al_i, \hspace{0.05in} \text{for} \hspace{0.05in} i=2,\ldots, k \} \subset V(G(\ast))$$ and let $$J = \{ (\al_1, \beta_2,\ldots,  \beta_k)  :  0 \leq \beta_i \leq \al_i, \text{for} \hspace{0.05in} i=2,\ldots, k \} \subset V(G(\ast)).$$ 
Consider the induced subgraphs $G_I$ and $G_J$ of $G(\ast)$. It is easy to see that $G_I \cap G_J = \emptyset$,  $V(G(\ast)) = V(G_I) \cup V(G_J)$,  $G_I \cong G(\al_1 -1, \al_2, \ldots, \al_k)$,  and $G_J \cong G(\al_2, \ldots, \al_k)$. Let $E_B = E(G(\ast)) \setminus \Big(E(G_I) \cup E(G_J)\Big)$. By construction, we have that $G(\ast) = G_I \cup G_J \cup E_B$. 

Now, $2$-cell embed $G_I$ on a surface $S_I$ of genus $\gamma(G(\al_1 -1, \al_2, \ldots, \al_k))$ and $2$-cell embed $G_J$ on a surface $S_J$ of genus $\gamma(G(\al_2, \ldots, \al_k))$. Since there is a bijective correspondence between the set of vertices of $G_J \cong G(\al_2, \ldots, \al_k)$ and the set of edges $E_B$, we have that $|E_B| = \prod_{i=2}^{k} (\al_i +1)$. In a small neighborhood of each vertex on $(G_J, S_J)$, attach one end of a handle (a copy of $\mathbb{S}^{1} \times [0,1]$) along with the corresponding edge from $E_B$ that connects to this vertex and attach the other end of the handle to the corresponding vertex on $(G_I, S_I)$. This will provide an embedding of $G(\ast)$ on a surface $S$, obtained from attaching $\Big[\prod_{i=2}^{k} (\al_i +1) \Big]$ handles between $S_I$ and $S_J$, and so, will have the desired genus. 
\end{proof}

Given any $k$-dimensional grid graph $G(\ast)$, Lemma \ref{lem:badgenusupperbound} could be applied successively to obtain an upper bound on $\gamma(G(\ast))$ in terms of the genera of subgraphs of $G(\ast)$, all of which are either $3$-dimensional grid graphs or grid graphs with at least three odd grid parameters. From here, one could apply Proposition \ref{prop:3dcase3} and  Theorem \ref{thm:kdimgenus1} to obtain an upper bound on $\gamma(G(\ast))$ that is exclusively a function of its grid parameters. As mentioned before the proof, this upper bound is  ineffective and could be improved by embedding more edges from $E_B$ on each handle. Determining an efficient scheme for distributing edges from $E_B$ onto handles would require more explicit details on the embeddings of $(G_I, S_I)$ and $(G_J, S_J)$.


\section{The Genus of a non-quadrilateral grid graph}
\label{subsec:lowerbounds}

We now discuss and apply   tools needed to determine a sharp lower bound on the genera of infinitely many  non-quadrilateral grid graphs. 

A \textbf{block} of a graph $G$ is a maximal $2$-connected subgraph $B$ of $G$. Given a connected graph $G$, there exists a unique collection of blocks $\mathcal{B} = \{ B_1, \ldots, B_k \}$ such that $\cup_{i=1}^{k} B_i = G$, which is called the \textbf{block decomposition} of $G$ in \cite{BHKY1962}.  The work of  Battle--Harary--Kodama--Youngs \cite{BHKY1962} provides a useful characterization of the genus of a graph in terms of its block decomposition. 

\begin{theorem}\cite{BHKY1962}
\label{thm:Blocks}
If $G$ is a graph with block decomposition  $ \{ B_1, \ldots , B_k \}$, then $\g(G) = \sum_{i=1}^k \g(B_i)$. Furthermore, if  $G$ is a graph with $m$ components $G_{1}, \ldots, G_{m}$, then $\g(G) = \sum_{i=1}^{m} \gamma(G_{m})$. 
\end{theorem}

The following graphs arise as minors of grid graphs analyzed in this section.   

\begin{definition}
\label{defn:Tn}
Let $T_n$ denote the graph constructed by gluing $n$ copies of $K_{3,3}$ along vertices as seen in Figure \ref{KnfromK33s}.
\end{definition}
\begin{figure}[h!]
\centering
\includegraphics[width=\textwidth]{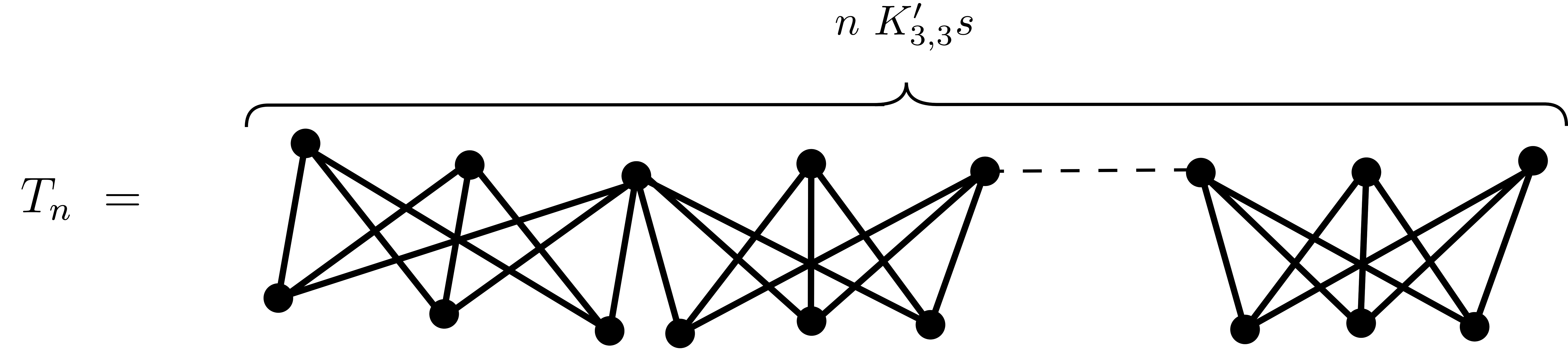}
\caption{}
\label{KnfromK33s}
\end{figure}

\begin{lemma}
\label{lem:Kngraph}
For each $n \in \mathbb{N}$, $T_n$ consists of $n$ blocks of $K_{3,3}$, and $\gamma(T_n) = n$. 
\end{lemma}

\begin{proof}
This lemma follows directly from Theorem \ref{thm:Blocks} and the fact that $\gamma(K_{3,3})=1$. 
\end{proof}


We now plan to use Lemma \ref{lem:Kngraph} to help find the genera of an infinite family of grid graphs that do not admit quadrilateral embeddings in Theorem \ref{thm:genus}. First, we need a few additional lemmas to deal with special cases that arise in the proof of Theorem \ref{thm:genus}.

\begin{figure}[ht]
\centering
\includegraphics[width=\textwidth]{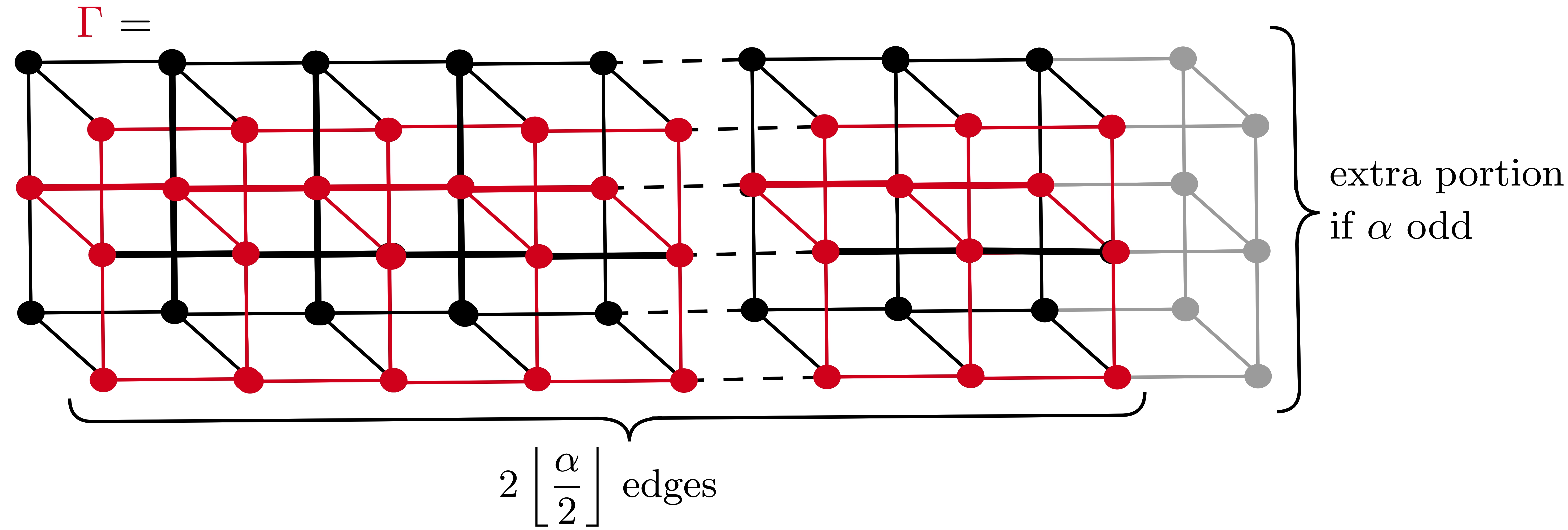}
\caption{Subgraph $\G$ of $G(\al,2,1)$ colored in red.}
\label{fig:SubG-in-Ga21}
\end{figure}

\begin{lemma}
\label{lem:genuslowerbound}
For any $\alpha \in \mathbb{N}$ where $\alpha\geq 2$, we have $\g(G(\al,2,1))  \geq \g(\Gamma) \geq \floor*{\frac{\alpha}{2}}$, where $\Gamma$ is the subgraph of $G(\al, 2, 1)$ shown in Figure \ref{fig:SubG-in-Ga21}.
\end{lemma}

\begin{proof}

Let $\alpha\in \N$ be arbitrary and let $a = \floor*{\frac{\alpha}{2}}$. The graph $G(\al, 2, 1)$ is depicted in Figure \ref{fig:SubG-in-Ga21} with a subgraph  $\Gamma$ highlighted. We can obtain a specific minor of $\G$ by contracting certain vertices (see Figure \ref{fig:Ka-in-subG}).

\begin{figure}[ht]
\centering
\includegraphics[width=\textwidth]{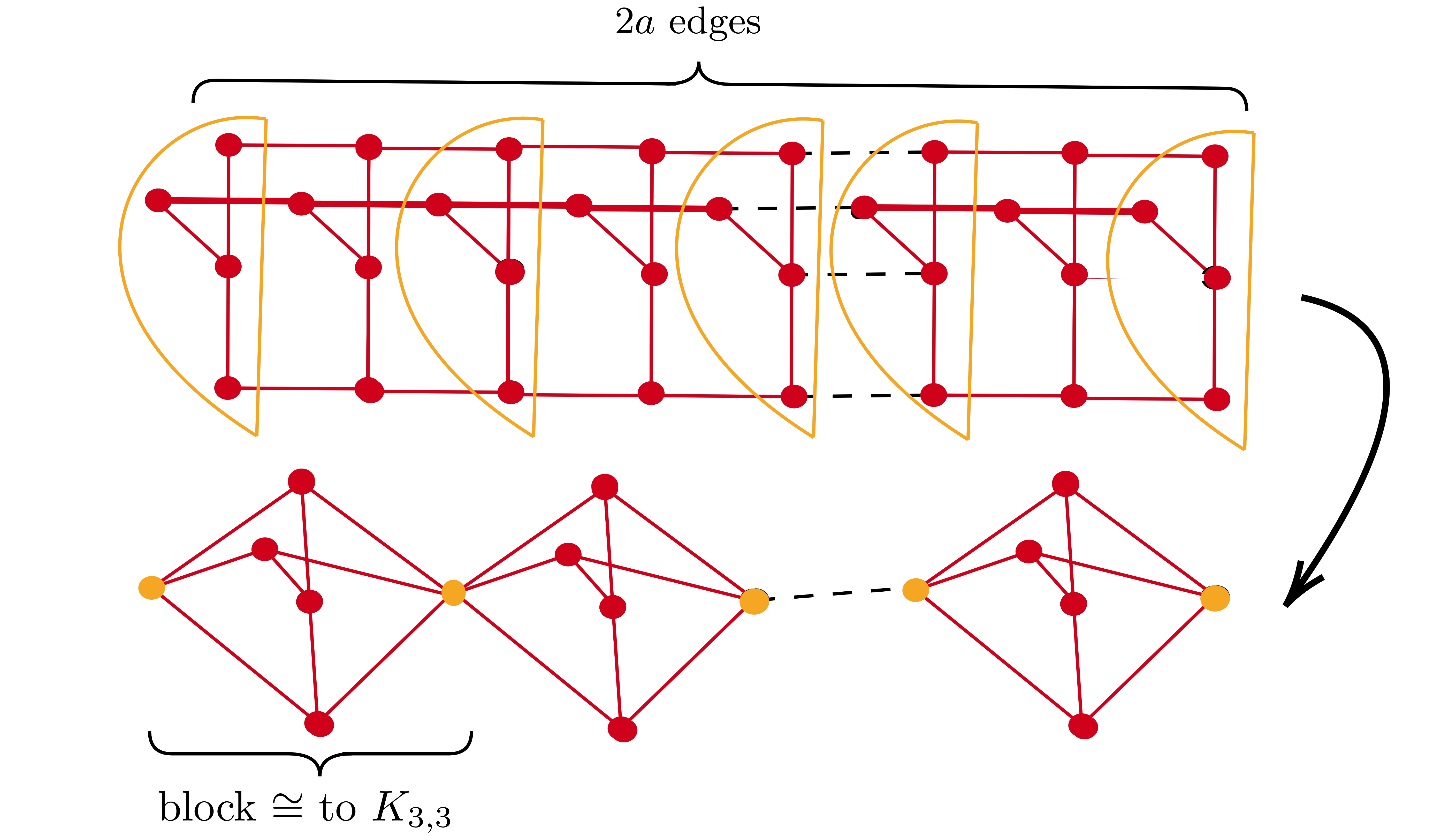}
\caption{Orange contracted vertices of $\G$ to get $T_a$ .}
\label{fig:Ka-in-subG}
\end{figure}
However, this minor is equivalent to $T_a$ as defined in Definition \ref{defn:Tn}. It then follows from Lemma \ref{lem:Kngraph} that $\g(T_a) = a$. So then we can deduce that $\g((G(\al, 2, 1)) \geq \g(\Gamma)\geq \g(T_a) = a$, as desired. 
\end{proof}

\begin{lemma}
\label{lem:Da11}
For any $\al \in \mathbb{N}$, $\gamma(G(\al,1,1)) = 0$, i.e., $G(\al, 1, 1)$ is planar.
\end{lemma}

\begin{proof}
Embed $G(\al, 1,1)$ in $\mathbb{R}^{3}$ with vertices positioned at their respective coordinates and edges connecting vertices via straight line segments. Then it is easy to see that $G(\al,1,1)$ is the $1$-skeleton  of a  $2$-cell structure of a cylinder with a square base. Take a homeomorphism of this cylinder in $\mathbb{R}^{3}$ to an annulus in $\mathbb{R}^{2}$. Then the image of $\Gamma$ under this homeomorphism will provide an embedding of $G(\al, 1, 1)$ in the plane. 
\end{proof}

\begin{theorem}
\label{thm:genus}
For any $\al_1,\al_2 \in \mathbb{N}$ we have that

\begin{equation*}
\gamma(G(\al_1, \al_2, 1)) = \floor*{\frac{\al_1}{2}} \floor*{\frac{\al_2}{2}}.
\end{equation*}

In addition, infinitely many $G(\al_1, \al_2, 1)$ do not admit quadrilateral embeddings. 
\end{theorem}

\begin{proof}
 To start,  if $\al_1= 1$ or $\al_2=1$, then $G(\al_1,\al_2,1) \cong G(\al_1,1,1)$ embeds in the plane by Lemma \ref{lem:Da11}. Thus,  $\g(G(\al_1,1,1))= 0 = \floor*{\frac{\alpha_1}{2}} \floor*{\frac{1}{2}}$, as needed. 

Now suppose $\al_1,\al_2\geq 2$. Using the subgraph $\G$ of $G(\al,2,1)$ described in Lemma \ref{lem:genuslowerbound}, we can glue together $\floor*{\frac{\al_2}{2}}$ copies of $G(\al,2,1)$ while taking care to keep the subgraphs $\G$ in each copy of $G(\al,2,1)$ disjoint, to get a subgraph of $G(\al_1,\al_2,1)$; see Figure \ref{fig:Glued-Da21-in-Dab1}. Since $\g(\G)\geq \floor*{\frac{\al_1}{2}}$ by Lemma \ref{lem:genuslowerbound} and we have $\floor*{\frac{\al_2}{2}}$ disjoint copies of $\G$ in $G(\al_1,\al_2,1)$, we can deduce that $\g(G(\al_1,\al_2,1))\geq \floor*{\frac{\al_1}{2}}\floor*{\frac{\al_2}{2}}$. 

\begin{figure}[ht]
\centering
\includegraphics[width=\textwidth]{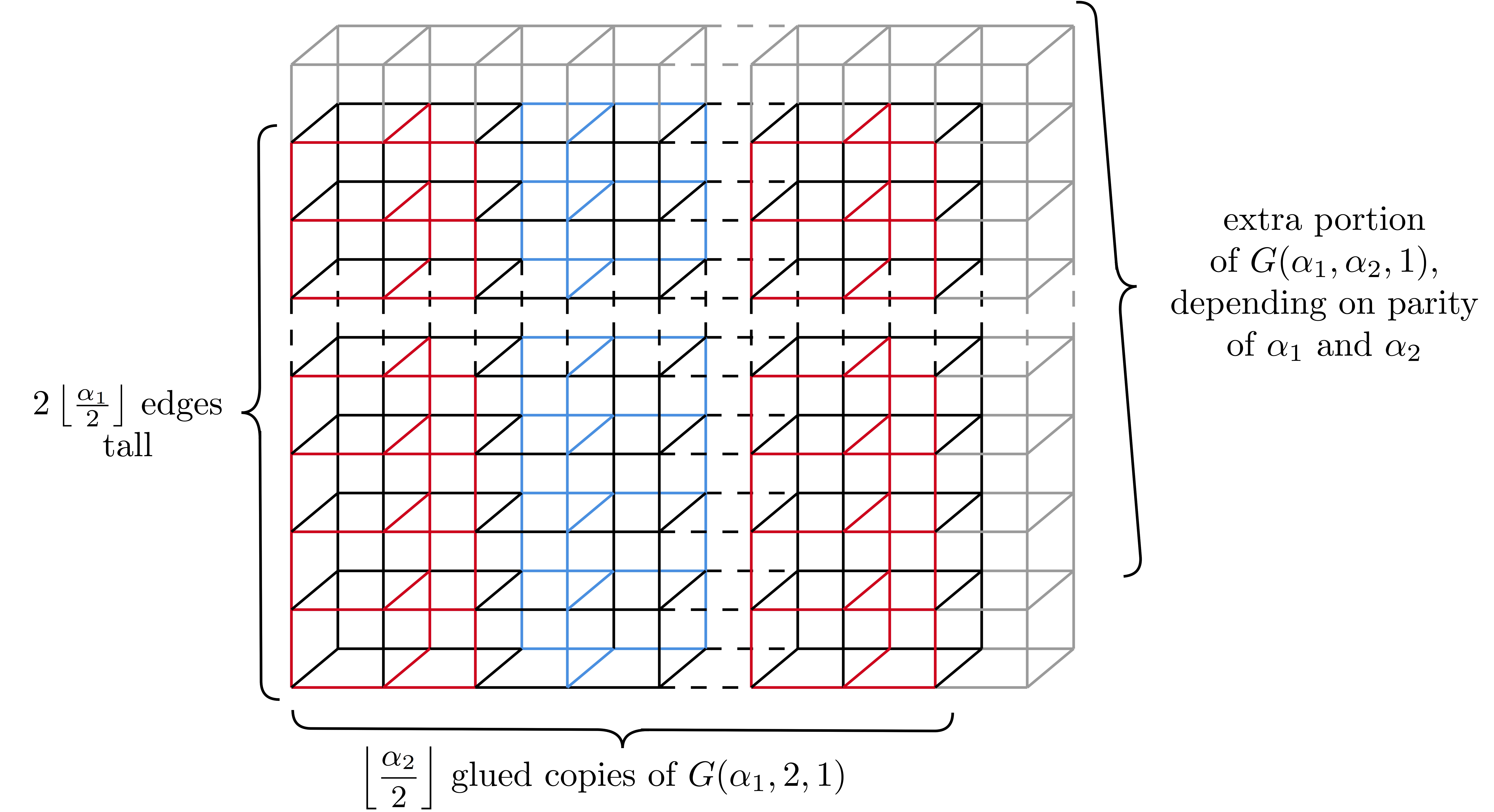}
\caption{Gluing of $\floor*{\frac{\al_2}{2}}$ copies of $G(\al_1,2,1)$ contained inside of $G(\al_1, \al_2, 1)$. Each copy of $G(\al_1, 2, 1)$ contains a subgraph isomorphic to $\Gamma$ (colored in red and blue) and these copies of $\Gamma$ are pairwise disjoint as subgraphs of $G(\al_1, \al_2, 1)$.}
\label{fig:Glued-Da21-in-Dab1}
\end{figure}

To show that $\g(G(\al_1,\al_2,1))\leq \floor*{\frac{\al_1}{2}}\floor*{\frac{\al_2}{2}}$, we can use Theorem \ref{thm:kdimgenus1} and Proposition \ref{prop:3dcase3}. This requires a few cases depending on the parity of $\al_1$ and $\al_2$. However, in all of these cases, the upper bounds on genus from  Theorem \ref{thm:kdimgenus1} and Proposition \ref{prop:3dcase3} are equivalent to $\floor*{\frac{\al_1}{2}}\floor*{\frac{\al_2}{2}}$. We leave the details of explicitly checking the algebra for the reader. 

It is easy to see that many of these grid graphs do not admit quadrilateral embeddings. Specifically, Corollary \ref{cor:Eulergirth} gives $\gamma(G(\al_1, \al_2, 1)) \geq \frac{1}{4}(\al_1 -1)(\al_2 -1)$, with equality if and only if the corresponding grid graph admits a  quadrilateral embedding, i.e., both $\al_1$ and $\al_2$ are odd in this case. Thus, if either $\al_1$ or $\al_2$ is even, then $G(\al_1,\al_2, 1)$ does not admit a quadrilateral embedding.
\end{proof}

The final paragraph of the proof of Theorem \ref{thm:genus}  shows that there exist grid graphs that are ``arbitrarily far away'' from having a quadrilateral embedding, which contrast the grid graphs studied by White in \cite{Wh1970}.  More generally, given a  graph $G$ with girth at least 4, define the \textbf{quadrilateral distance} of $G$ to be the quantity $$d_{Q}(G) = \gamma(G) - \Big(1 + \frac{|E(G)|}{4} - \frac{|V(G)|}{2}\Big).$$   Note that $d_{Q}(G) \geq 0$ and  $d_{Q}(G) = 0$ if and only if $G$ admits a quadrilateral embedding. The following corollary immediately follows from Theorem \ref{thm:genus}.

\begin{corollary}
Consider the sequence of grid graphs $\{ G(2n,2,1) \}_{n=1}^{\infty}$. Then $$d_{Q}(G(2n,2,1)) \rightarrow \infty \hspace{0.05in} \text{as} \hspace{0.05in} n \rightarrow \infty.$$ 
\end{corollary}

We now determine the genera of a different infinite family of non-quadrilateral grid graphs.

\begin{theorem}
\label{thm:Da22}
For any $\alpha\in\N$, we have that $\gamma(G(\al, 2,2)) = \alpha$. Every $G(\al, 2,2)$ does not admit a quadrilateral embedding. 
\end{theorem}

\begin{proof}
Figure \ref{fig:Da22-contract-to-K3,4a} shows how to delete and contract edges of $G(\alpha,2,2)$ to obtain a $K_{3,4\alpha}$ minor. By the genus formula for complete bipartite graphs (\cite{Ri1965}, \cite{Bo1978})
we have:
$$\alpha = \left\lceil \alpha - \frac{1}{2}\right\rceil=\left\lceil\frac{(3-2)(4\alpha-2)}{4}\right\rceil = \gamma(K_{3,{4\alpha}}) \leq \gamma(G(\alpha,2,2)).$$

\begin{figure}[ht]
\centering
\includegraphics[width=\textwidth]{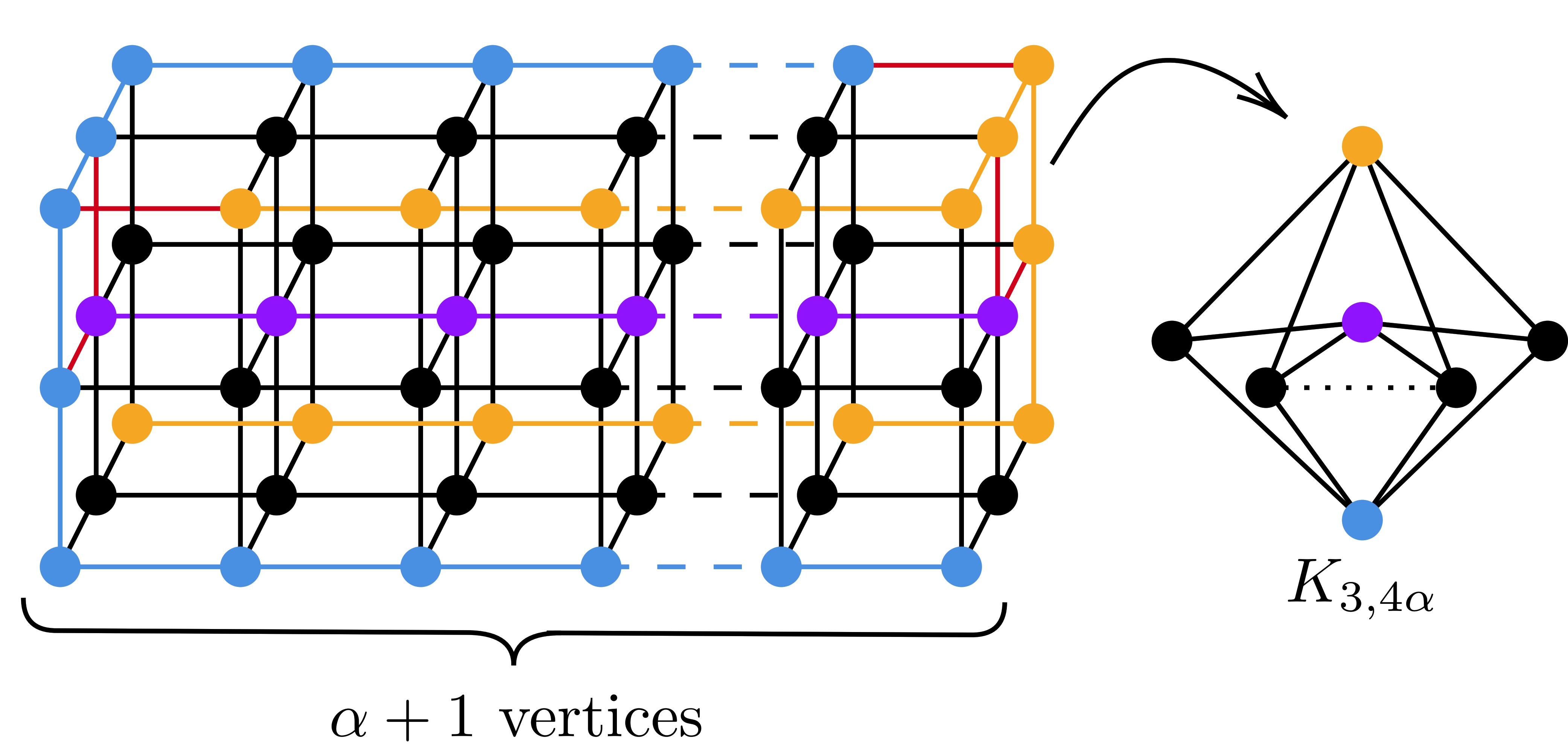}
\caption{Minor of $G(\alpha,2,2)$ to $K_{3,4\alpha}$. The blue, orange, and purple paths are contracted to single vertices. The red edges are deleted.}
\label{fig:Da22-contract-to-K3,4a}
\end{figure}

To obtain the desired upper bound on any $\gamma(G(\al, 2, 2))$, we apply results (2) and (3) from Proposition \ref{prop:3dcase3}, depending on whether $\al$ is even or odd, respectively. First, recall that $G(\al, 1, 1)$ is planar by Lemma \ref{lem:Da11}. Now, if $\al$ is odd, then 
$$ \gamma(G(\al,2,2)) \leq \gamma(G(\al, 1, 1)) + \frac{(\al +1)(2+2)}{4} -1  = \al.$$
If $\al$ is even, then 
$$\gamma(G(\al,2,2)) \leq \gamma(G(\al-1, 1, 1)) + \frac{2\al  + 2\al + 4}{4} -1  = \al.$$

Finally, we note that  $G(\alpha,2,2)$ is not quadrilateral  for any $\al \in \mathbb{N}$ since the lower bound on its genus from Corollary \ref{cor:Eulergirth} gives $\gamma(G(\al, 2, 2)) \geq \frac{3\al -2}{4}$, and $\gamma(G(\al, 2,2)) = \al > \frac{3\al -2}{4}$ for all $\al \in \mathbb{N}$. 
\end{proof}

Consider the sets of grid graphs  $\{G(\al, 2,2)\}$ with $\al \in \mathbb{N}$ and $\{G(\al_1, \al_2, 1)\}$ with some $\al_i$  even for $i=1,2$. All of these $3$-dimensional grid graphs have at least one even grid parameter and a minimal genus $2$-cell embedding for each such graph can be constructed using the procedure described in the proof of Proposition \ref{prop:3dcase3}. This motivates the following conjecture:

\begin{conjecture}
\label{conjecture1}
If $G(\al_1, \al_2, \al_3)$ is a $3$-dimensional grid graph with at least one even parameter, then the $2$-cell embedding for $G(\al_1, \al_2, \al_3)$ constructed in Proposition \ref{prop:3dcase3} is a  minimal embedding, i.e., the upper bounds in Proposition \ref{prop:3dcase3} are in fact equalities.  
\end{conjecture}

Perhaps a generalization of the  arguments used in the proofs of Theorem \ref{thm:genus} and Theorem \ref{thm:Da22} could be modified to prove this conjecture. However, this was not immediately obvious  since different arguments were used to determine sharp lower bounds on  $\gamma(G(\al_1, \al_2, 1))$ and $\gamma(G(\al, 2,2))$. We tried to find a block decomposition that would provide a sharp lower bound on  $\gamma(G(\al,2,2))$ but were unsuccessful.


\section{The Maximum Genus of a Grid Graph}
\label{sec:GenusRange}
The \textbf{maximum genus} $\gamma_{M}(G)$  of a  graph $G$ is the largest possible $g$ for which $G$ $2$-cell embeds on $S_g$.  Recall that the \textbf{Betti number} of a graph $G$ is $\beta(G) = |E| - |V| + 1$. If $G$ $2$-cell embeds on $S_g$, we have that $2-2g = \chi(S_g) = |V| - |E| + |F|.$  Any such embedding must have at least one face, and so, we always have the upper bound $\gamma_{M}(G) \leq  \lfloor{\frac{\beta(G)}{2}} \rfloor$. A graph $G$ is called \textbf{upper-embeddable} if $\gamma_{M}(G) =  \lfloor{\frac{\beta(G)}{2}} \rfloor$.  Here, we  summarize some facts about the maximum genus of a grid graph. 

\begin{theorem}
\label{thm:UEmaintheorem}
Every grid graph is upper-embeddable. 
\end{theorem}

\begin{proof}
First,  any $1$-dimensional grid graph $G(\al)$ is upper-embeddable since such a graph is a path with $\al + 1$ vertices and $\al$ edges, and we obviously have that $\gamma_{M}(G(\al)) = 0$. 

Now, Corollary 3.3 in Mohar--Pisanski--\u Skoviera \cite{MoPiSk1988} states that the Cartesian product of any two loopless connected non-trivial graphs is upper-embeddable.  Note that any $k$-dimensional grid graph  with $k \geq 2$ can be expressed as the Cartesian product of two subgraphs that are each (lower-dimensional) grid graphs. At the same time, any grid graph is a loopless connected non-trivial graph. Thus, by \cite[Corollary 3.3]{MoPiSk1988}, any grid graph is upper-embeddable.
\end{proof}

The following corollary  characterizes the maximum genus of a grid graph in terms of its grid parameters.

\begin{corollary}
\label{thm:MaxGenusGrids}
The maximum genus of a $k$-dimensional grid graph $G(\al_1, \ldots, \al_k)$ is exactly
$$ \gamma_{M}(G(\al_1, \ldots, \al_k)) = \Bigg\lfloor{ \frac{1}{2}  \left(\prod_{i=1}^{k} (\alpha_{i} +1)\right) \Big( \left(\sum_{i=1}^k \frac{\al_i}{\al_i+1} \right)-1 \Big) +  \frac{1}{2}  } \Bigg\rfloor $$
\end{corollary}

\begin{proof}
 $G(\al_1, \ldots, \ldots \al_k)$ is upper-embeddable by Theorem \ref{thm:UEmaintheorem}, which implies that $$\gamma_{M}(G(\al_1, \ldots, \al_k)) = \Big\lfloor{\frac{\beta(G(\al_1, \ldots, \al_k))}{2}} \Big\rfloor.$$ We get the desired formula by applying Lemma  \ref{prop:verticesforgrids} to obtain formulas for $|E|$ and $|V|$, respectively,   in terms of the grid parameters $\al_1, \ldots, \al_k$.
\end{proof}

The \textbf{genus range} of a graph $G$ is the set $R(G) = \{g \in \mathbb{N} :  \gamma(G) \leq g \leq  \gamma_{M}(G) \}$. The interpolation theorem of Duke \cite{Du1966} shows that $G$ will $2$-cell embed on every surface $S_g$, where $g \in R(G)$. One could combine Corollary \ref{thm:MaxGenusGrids} with  results from Section \ref{sec:BlockandGenus} to obtain explicit formulas for the genus range of many grid graphs in terms their respective grid parameters. To give an explicit example that will also be used in the proof of Theorem \ref{thm:gridtorus}, we consider the genus range of any $k$-dimensional $G(1, \ldots, 1)$. Such a grid graph is isomorphic to the well studied $k$-dimensional hypercube, $Q_k$. An exact genus formula for $Q_k$ was first proved by Ringel in \cite{Ri1955} and Beineke and Harary in \cite{BeHa1965}. More recently, Shouman \cite{Sh2019} has provided a proof using the real moment-angle complex. In addition, a maximum genus formula for $Q_k$ was  given by Zaks in \cite{Za1974}. A genus formula for $Q_k$ can be derived as a special case of Theorem \ref{thm:kdimgenus1}, while Corollary \ref{thm:MaxGenusGrids}  supplies a maximum genus formula, both of which are highlighted in the following corollary.

\begin{corollary}
\label{cor:hypercuberange}
For any $k \geq 2$, the genus range of the $k$-dimensional grid graph $G(1,\ldots, 1) \cong Q_k$ is exactly $$\{ g \in \mathbb{N} :  1+(k-4)2^{k-3} \leq g \leq (k-2)2^{k-2} \}.$$ 
\end{corollary}


\section{Planar and Toroidal Grid Graphs}
\label{sec:planarandtoroidal}

In this section, our main goal is to classify which grid graphs $2$-cell embed on the torus. First, we classify which grid graphs are planar.

\begin{proposition}
A grid graph $G(\al_1, \ldots, \al_k)$ is planar if and only if
\begin{enumerate}
\item $k \leq 2$ or
\item $k =3$ and at most one of $\al_1$, $\al_2$, $\al_3$ is greater than $1$.
\end{enumerate}
\label{prop:PlanarGrid}
\end{proposition}

\begin{proof}
The fact that $k$-dimensional grid graphs are planar for $k =1,2$ is obvious. For the second case, any such grid graph is isomorphic to $G(\al,1,1)$  for some $\al \in \mathbb{N}$. Such grids are planar by  Lemma \ref{lem:Da11}. 

To show that no other grid graphs are planar, we consider two cases. Any $3$-dimensional grid graph with at least two of $\al_1$, $\al_2$, and $\al_3$ greater than 1 contains a subgraph isomorphic to $G(2, 2, 1)$. However, Theorem \ref{thm:genus} shows that $\gamma(G(2,2,1)) = 1$, and so, any graph containing $G(2,2,1)$ as a subgraph is non-planar.  Now consider any $k$-dimensional grid graph where $k \geq 4$. Any such grid graph contains a subgraph isomorphic to $G(1,1,1,1)$. Using Corollary \ref{cor:hypercuberange}, we see that $G(1,1,1,1)$ is non-planar, and so, any graph containing $G(1,1,1,1)$ as a subgraph is non-planar, as needed. 
\end{proof}

\begin{theorem}
\label{thm:gridtorus}
A grid graph $G(\alpha_1, \ldots, \alpha_k)$ $2$-cell embeds on the torus if and only if
\begin{enumerate}
\item $k =2$ and $\alpha_{1} + \alpha_{2} \geq 3$,
\item $k=3$ and at most one of $\alpha_{1}$, $\alpha_{2}$, $\alpha_{3}$ is greater than 1,
\item $k=3$ and  up to permutation, $(\alpha_{1}, \alpha_{2}, \alpha_{3}) \in \{ (2,2,1), (3,2,1), (3,3,1) \} $,   or
\item $k=4$ and $(\alpha_{1}, \alpha_{2}, \alpha_{3}, \alpha_{4}) = (1,1,1,1)$.
\end{enumerate}
\end{theorem}

\begin{proof}

We break this proof down into a series of cases, based on the dimension $k$. 

For $k=1$, we have that $G(\alpha_1)$ is a path, which does not $2$-cell embed on the torus.

For $k=2$,  we can apply Corollary \ref{thm:MaxGenusGrids} to get that the genus range of $G(\al_1, \al_2)$ is precisely $\{ g \in \mathbb{N} :  0 \leq g \leq \lfloor{\frac{1}{2}\al_{1}\al_{2} } \rfloor \}.$ Thus, $G(\al_1, \al_2)$ embeds on the torus if and only if $\al_1 + \al_2 \geq 3$, as needed. 

For $k=3$, multiple subcases are necessary. If at most one of $\alpha_{1}$, $\alpha_{2}$, $\alpha_{3}$ is greater than $1$, then  such a grid graph is graph isomorphic to $G(\al, 1, 1)$ for some $\al \in \mathbb{N}$, which is planar by Lemma \ref{lem:Da11}. To determine that such a graph $2$-cell embeds on the torus, apply Corollary \ref{thm:MaxGenusGrids} to establish that $\gamma_{M}(G(\al, 1,1)) = \lfloor{ \frac{4\al+1}{2} \rfloor}$. Thus, the genus range for such graphs always include $g=1$, as needed. Next, the grid graphs $G(2,2,1)$,  $G(3,2,1)$, and $G(3,3,1)$ all $2$-cell embed on the torus by Theorem \ref{thm:genus}.  Now, we will justify that any other $3$-dimensional grid graph does not $2$-cell embed on the torus. By Theorem \ref{thm:Da22},  $\gamma(G(2, 2, 2)) =2$. Then if all $\al_i \geq 2$, we have that  $\gamma(G(\al_1, \al_2, \al_3)) \geq \gamma(G(2,2,2)) \geq 2$, and so, any such $G(\al_1, \al_2, \al_3)$ does not $2$-cell embed on the torus.  Thus, we now only need to consider $G(\alpha_{1}, \alpha_{2}, \alpha_{3})$ where exactly one such $\alpha_{i} =1$.  Combined with the fact we just showed $G(2,2,1)$, $G(3,2,1)$, and $G(3,3,1)$ are all toroidal,  it only remains to consider $G(\alpha_{1}, \alpha_{2}, 1)$ with $\alpha_{1} \geq 4$ and $\alpha_{2} \geq 2$. By Theorem \ref{thm:genus}, we have that $\gamma(G(4,2,1)) =2$, and so, $G(4,2,1)$ does not $2$-cell embed in the torus. Thus, the same conclusion holds for any such $G(\alpha_{1}, \alpha_{2}, 1)$ with $\alpha_{1} \geq 4$ and $\alpha_{2} \geq 2$.

For $k=4$,  $G(1,1,1,1)$ is isomorphic to the four dimensional hypercube $Q_4$, which $2$-cell embeds on the torus by Corollary \ref{cor:hypercuberange}.  Any other grid graph in this case contains a subgraph isomorophic to $G(2,1,1,1)$, which does not $2$-cell embed on the torus by Corollary \ref{cor:Eulergirth}. Thus, the same conclusion holds for any $4$-dimensional grid graph left in this case. 

For $k \geq 5$, every such grid graph contains a subgraph isomorphic to $G(1,1,1,1,1)$. However, Corollary \ref{cor:Eulergirth} implies that $G(1,1,1,1,1)$ does not $2$-cell embed in the torus, and so, the same holds for all other grid graphs of dimension $k \geq 5$. \end{proof}

We note that there are only trivial examples of grid graphs that embed on the torus but do not $2$-cell embed on the torus, specifically, any $G(\al_1, \ldots, \al_k)$ with $k \leq 2$ will meet these qualifications. 

Using the tools we developed in Section \ref{sec:BlockandGenus} along with White's work \cite{Wh1970}, one could possibly continue to classify which grid graphs minimally embed  on a fixed surface of genus $g \in \mathbb{N}$. In the remainder of this section, we just show  that the set of grid graphs that minimally embed on a fixed surface $S_g$ is always finite for $g > 0$.

Define $\mathscr{C}_g$ as the set of grid graphs that embed on $S_g$ but not on $S_{g-1}$. Proposition \ref{prop:PlanarGrid} gives a complete description of $\mathscr{C}_{0}$ and one could easily use  Theorem \ref{thm:gridtorus}  to give a complete description of $\mathscr{C}_{1}$. The following theorem and its proof provide some evidence for how $\mathscr{C}_{g}$ behaves more generally. 

\begin{theorem}
The set $\mathscr{C}_g$ is finite if and only if $g\neq 0$. Equivalently, the collection of grid graphs that minimally embed on any closed orientable surface with genus $g$  is infinite, if and only if $g=0$.
\end{theorem}

\begin{proof}
Since any graph contained in $\mathscr{C}_g$ corresponds with a minimal embedding, we don't distinguish between embeddings and $2$-cell embeddings in this proof.  We can easily see that $\mathscr{C}_0$ is infinite since for $k=1,2$,  every $k$-dimensional grid graph embeds in the plane.

Now, suppose there exists $g_0 \in \mathbb{N}$ such that $| \mathscr{C}_{g_{0}} | =  \infty$, say $\mathscr{C}_{g_{0}} = \{ G_i \}_{i=1}^{\infty}$, where each $G_i$ is a grid graph with genus  $\gamma(G_i) = g_0$.  Then $\{ G_i \}_{i=1}^{\infty}$ contains a subsequence $\{ G_{i_{j}} \}$ where 
\begin{enumerate}
\item the dimension of $G_{i_{j}}$ strictly increase to $\infty$ as $j \rightarrow \infty$ or,
\item there are infinitely many $G_{i_{j}}$ with the same dimension. 
\end{enumerate}
We will now show that both cases are impossible, giving the desired contradiction.

For (1), let $\{ G_k \}$ be such a subsequence, where each $k$ designates the dimension of the corresponding grid graph. So, $k \rightarrow \infty$ by assumption. Note that, we can assume $k \geq 3$ here since any $k$-dimensional grid graph with $k \leq 2$ is planar. Now, each $G_k$  contains a subgraph isomorphic to $G(1, \ldots, 1)$ with $k$ ones  and Corollary \ref{cor:hypercuberange} implies that 
$$\gamma(G_k)  \geq \gamma(G(1, \ldots, 1)) = 1 + (k-4)2^{k-3}.$$

Set $f(k) = 1 + (k-4)2^{k-3}$. Since $\displaystyle\lim_{k \rightarrow \infty} f(k) = \infty$, there exists some $k'$ such that $g_{0} \leq f(k)$ for all $k  \geq k'$. This implies that any $k$-dimensional grid graph could not embed on $S_{g_{0}}$ for $k \geq k'$. This completes case (1). 

For (2), suppose we have a sequence of $k$-dimensional grid graphs, $\{G(\alpha_{j_1}, \ldots, \alpha_{j_k}) \}_{j=1}^{\infty}$ for some fixed $k \geq 3$. Then this sequence contains a subsequence where one of the $\alpha_{j_{m}}$ parameters heads to infinity as $j \rightarrow \infty$. Without loss of generality, assume this parameter is $\alpha_{j_{1}}$. Then each grid graph in this further subsequence contains a subgraph isomorphic to the $k$-dimensional grid graph $G(\alpha_{j_{1}}) = G(\alpha_{j_{1}}, 1, \ldots, 1)$. Then Corollary \ref{cor:Eulergirth} implies that

$$ \begin{aligned}  \gamma(G(\alpha))   &\geq  1 - (\alpha+1)2^{k-2}\Big[1- \frac{1}{2}(\frac{\alpha}{\alpha+1} + \frac{k-1}{2}) \Big] \\
 &= 1-2^{k-2}\Big[\alpha+1 - \frac{1}{2}(\alpha + \frac{(\alpha+1)(k-1)}{2}) \Big] \\
& = 1 - 2^{k-2} \Big[ \frac{3}{4}\alpha + \frac{5}{4} - \frac{k}{4}(\alpha-1)\Big] \\
&= 1 + 2^{k-2} \Big[ (\frac{k-3}{4})\alpha - \frac{5+k}{4} \Big].
\end{aligned} $$

Now, for any fixed $k > 3$, let $h(\al_{j_{1}}) = 1 + 2^{k-2} \Big[ (\frac{k-3}{4})\al_{j_{1}} - \frac{5+k}{4} \Big].$  Then we have that $\displaystyle\lim_{j \rightarrow \infty} h(\alpha_{j_{1}}) = \infty$, implying that (2) can not hold when the dimension is $k > 3$ following the same line of argument as case (1). Finally, for $k = 3$, since Lemma \ref{lem:Da11} implies that  $G(\alpha, 1, 1)$ is planar, we can conclude that any $3$-dimensional grid graph in this sequence must contain a subgraph isomorphic to $G(\alpha, 2, 1)$. However, Lemma \ref{lem:genuslowerbound} tells us that $\g(G(\al,2,1))  \geq  \floor*{\frac{\alpha}{2}}$. Since $\displaystyle\lim_{j \rightarrow \infty}\floor*{\frac{\alpha_{j_{1}}}{2}} = \infty$, we can again apply the same type of argument used in case (1). Thus, such a sequence is also impossible when $k=3$.
\end{proof}


\begin{acknowledgements}
 We would  like to thank Mark Ellingham and the referee for providing useful feedback on drafts of this paper. \end{acknowledgements}

\section*{Funding}
This work was financially supported by the Furman University Department of Mathematics  via the Summer Mathematics Undergraduate Research Fellowships.

\section*{Availability of data and materials}
Not applicable. 

\section*{Conflict of Interests}
The authors have no relevant financial or non-financial interests to disclose.

\bibliographystyle{hamsplain}
\bibliography{biblio}

\end{document}